\documentclass[12pt,reqno]{article}
\usepackage{fullpage}
\usepackage{float}
\usepackage{booktabs}
\usepackage{enumerate}
\usepackage{amsthm,amsmath,amssymb}
\usepackage{tikz}
\usepackage{multirow,multicol}
\usepackage[export]{adjustbox}
\definecolor{webgreen}{rgb}{0,.5,0}
\definecolor{webbrown}{rgb}{.6,0,0}
\usepackage[colorlinks=true,
linkcolor=webgreen,
filecolor=webbrown,
citecolor=webgreen]{hyperref}
\usepackage{subfigure}
\newcommand{\seqnum}[1]{\href{http://oeis.org/#1}{\underline{#1}}}
\newcommand{\Figs}[1]{\hyperref[#1]{Figure~\ref*{#1}}}
\newcommand{\Tabs}[1]{\hyperref[#1]{Table~\ref*{#1}}}
\newcommand{\Sec}[1]{\hyperref[#1]{section~\ref*{#1}}}
\newcommand{\Lem}[1]{\hyperref[#1]{Lemma~\ref*{#1}}}
\everymath{\displaystyle}
\providecommand{\tabularnewline}{\\}
\theoremstyle{plain}
\newtheorem{theorem}{Theorem}
\newtheorem{lemma}[theorem]{Lemma}

\theoremstyle{definition}

\newtheorem{notation}[theorem]{Notation}

\theoremstyle{remark}
\newtheorem{remark}[theorem]{Remark}

\title{\bf A one-variable bracket polynomial for some  Turk's head knots}
\author{Franck Ramaharo\\
\small D\'epartement de Math\'ematiques et Informatique\\[-0.8ex]
\small Universit\'e d'Antananarivo\\[-0.8ex] 
\small 101 Antananarivo, Madagascar\\
\small\href{mailto:franck.ramaharo@gmail.com}{\tt franck.ramaharo@gmail.com}\\
}
\date{\small\today\\}

\begin{document}

\maketitle

\begin{abstract}
We compute the Kauffman bracket polynomial of \textit{the three-lead Turk's head}, \textit{the chain sinnet} and \textit{the figure-eight chain} shadow diagrams. Each of these knots can in fact  be constructed by repeatedly concatenating  the same $ 3 $-tangle, respectively, then taking the closure. The bracket is then evaluated by expressing the state diagrams of the concerned $ 3 $-tangle by means of the Kauffman monoid diagram's elements.

\bigskip\noindent  {Keywords:} bracket polynomial, tangle shadow, Kauffman state, flat sinnet.
\end{abstract}

\section{Introduction}
The present paper is a follow-up on our previous work which aims at collecting statistics on knot shadows \cite{Ramaharo}. We would like to establish  the bracket polynomial for knot diagram generated by the  $ 3 $-tangle shadows below:
\begin{equation}\label{eq:braid}
\begin{array}{cccccccc}
\protect\includegraphics[width=0.075\linewidth,valign=c]{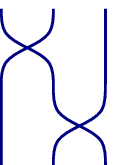}&\qquad&
\protect\includegraphics[width=0.075\linewidth,valign=c]{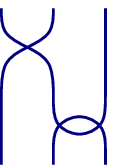}&\qquad&
\protect\includegraphics[width=0.075\linewidth,valign=c]{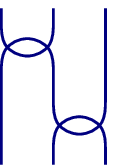}\\
T  &\qquad&  C  &\qquad&  E
\end{array}.
\end{equation}
The knot diagrams under consideration are those obtained by repeatedly multiplying (or concatenating) the same $ 3 $-tangle, then taking the closure of the resulting $ 3 $-tangle (i.e., connecting the endpoints in a standard way, without introducing further crossings between the strands).  Knots obtained from the $ 3 $-tangles pictured in \eqref{eq:braid} belong to the Ashley's \textit{Turk's head} family \cite[p.\ 226, Chap.\ 17]{ABOK}: the \textit{three-lead Turk's head} \cite[\#1305]{ABOK}, the \textit{chain sinnet} \cite[\#1374]{ABOK} and the \textit{figure-eight chain} \cite[\#1376]{ABOK}, respectively (e.g. see \Figs{Fig:TurksHead}).

\begin{figure}[ht]
\centering
$\begin{array}{cccccccc}
\protect\includegraphics[width=0.225\linewidth,valign=c]{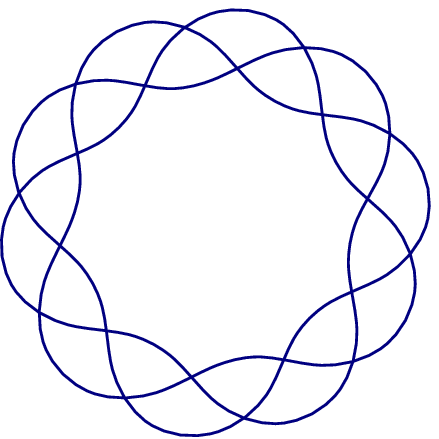}&\quad&
\protect\includegraphics[width=0.225\linewidth,valign=c]{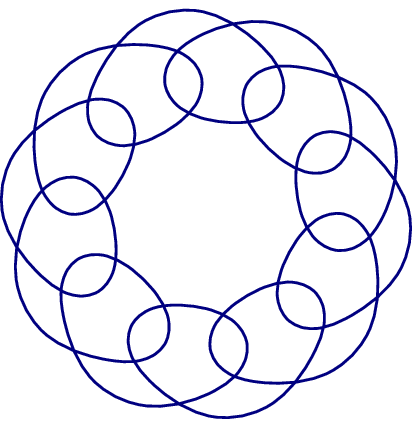}&\quad&
\protect\includegraphics[width=0.225\linewidth,valign=c]{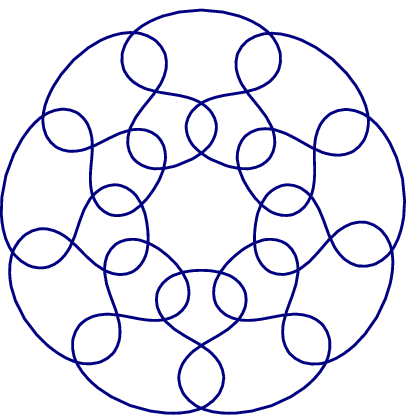}\\
\mbox{Three-lead Turk's head}  &\quad& \mbox{Chain sinnet}  &\quad&  \mbox{Figure-eight chain}
\end{array}$
\caption{Some flat Turk's-head knot diagrams.}
\label{Fig:TurksHead}
\end{figure}

The remainder of this paper is arranged as follows. In \Sec{sec:bracket}, we establish the expression of the bracket polynomial for any $ 3 $-tangle shadow diagram. Then in \Sec{sec:application}, we apply those results to the flat sinnet Turk's heads  mentioned earlier.

\section{The Kauffman bracket of a \texorpdfstring{$3$}{3}-tangle shadow}\label{sec:bracket}
In this paper, the Kauffman bracket maps a shadow diagram $ D $ to $ \left< D\right>\in$  $\mathbb{Z}[x] $ and is constructed  from the following rules:
\begin{itemize}
\item [$ (\mathbf{K1}) $:] $  \left<\bigcirc \right>=x $;\label{it:i}
\item [$ (\mathbf{K2}) $:]  $ \left<\bigcirc\sqcup D\right>=x\left< D\right>$;\label{it:ii}
\item [$ (\mathbf{K3}) $:]  $\left<\protect\includegraphics[width=.03\linewidth,valign=c]{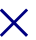}\right>=\left<\protect\includegraphics[width=.03\linewidth,valign=c]{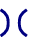}\right>+\left<\protect\includegraphics[width=.035\linewidth,valign=c]{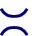}\right>$. \label{it:iii}
\end{itemize}

The diagram $ \bigcirc $ in $ (\mathbf{K1}) $ represents that of a single loop, and the symbol $ \sqcup $  in $ (\mathbf{K2}) $ denotes the disjoint union operation. Formula in $ (\mathbf{K3}) $ expresses the splitting of a crossing. Recall that the choice of such splittings for any single crossing is referred to as the so-called \textit{Kauffman state}.  Rules $ (\mathbf{K1}) $, $ (\mathbf{K2}) $ and $ (\mathbf{K3}) $ can be summarized by the summation which is taken over all the states for $ D $, namely $ \left<D\right>=\sum_{S}^{}x^{|S|} $ ,  where $ |S| $ gives the number of loops in the state $ S $. Kauffman shows that the states elements of  a $ 3$-tangle diagram $ B :=\protect\includegraphics[width=0.05\linewidth,valign=c]{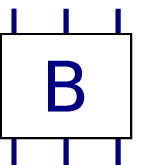}$ are generated by the product of a loop and the  following $ 5 $ elements of the \textit{$ 3 $-strand diagram monoid}  $ \mathcal{D}_3 $ \cite{Kauffman1,Stoimenow}:
\[\begin{array}{ccccccccc}
\protect\includegraphics[width=0.07\linewidth,valign=c]{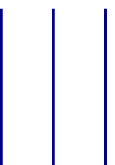}&\qquad&
\protect\includegraphics[width=0.075\linewidth,valign=c]{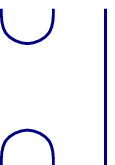}&\qquad&
\protect\includegraphics[width=0.075\linewidth,valign=c]{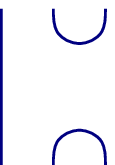}&\qquad&
\protect\includegraphics[width=0.075\linewidth,valign=c]{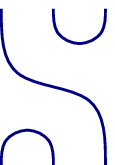}&\qquad&
\protect\includegraphics[width=0.075\linewidth,valign=c]{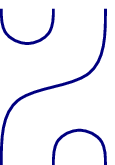} \\
1_3  &\qquad&  U_1  &\qquad&  U_2 	&\qquad&  r &\qquad&  s 
\end{array}.\]
In other words, given a state $ S $, there exist a nonnegative integer $ k $ and an element $ U $ in $ \mathcal{D}_3 $ such that one writes $ S=\bigcirc^k\sqcup U $, where $ \bigcirc^k=\bigcirc\sqcup\bigcirc\sqcup\cdots\sqcup\bigcirc $ denotes the disjoint union of $ k $ loops \cite[p.\ 100]{Kauffman2}. The bracket of the $ 3$-tangle $ B$ becomes  $ \left<B\right>=\sum_{S}^{}\left<S\right>$, where $ \left<S\right>=x^{|S|}\left<U\right> $ for certain $U\in\mathcal{D}_3  $. 

Therefore $ \left<B\right> $ is a linear combination of the brackets $ \left<1_3\right> $, $\left<U_1\right> $, $\left<U_2\right> $,$\left<r\right> $ and $\left<s\right> $, i.e., there exist five polynomials $ a,b,c,d,e $ in $ \mathbb{Z}[x] $  such that \begin{equation}\label{eq:bracket}
\left<B\right>=a\left<1_3\right>+b\left<U_1\right>+c\left<U_2\right>+d\left<r\right>+e\left<s\right>.
\end{equation} 

\begin{lemma}\label{lem:product}
Given two $ 3 $-tangles $ B$ and $ D $, we have
\begin{align*}
\left<BD\right>&=a_Ba_D\left<1_3\right>+\left(b_Ba_D+\left(a_B+b_B x+d_B\right)b_D+\left(d_B x+b_B\right)e_D\right)\left<U_1\right>\\
&\hphantom{=}+\left(c_Ba_D+\left(a_B+c_B x+e_B\right)c_D+\left(c_B+e_B x\right)d_D\right)\left<U_2\right>\\
&\hphantom{=}+\left(d_Ba_D+\left(d_B x+b_B\right)c_D+\left(a_B+b_B x+d_B\right)d_D\right)\left<r\right>\\
&\hphantom{=}+\left(e_Ba_D+\left(c_B+e_B x\right)b_D+\left(a_B+c_B x+e_B\right)e_D\right)\left<s\right>.
\end{align*}.
\end{lemma}

\begin{proof}
We first establish the states of $ B $ leaving $ D $ intact, and then in $ D $:
\begin{align*}
\left<BD\right>&= a_B  a_D \left<1_3^2\right>+ a_B  b_D \left<1_3   U_1\right> + a_B  c_D \left<1_3   U_2\right>+ a_B  D \left<1_3   r\right>+ a_B  e_D \left<1_3   s\right>\\
&\hphantom{=}+ b_B  a_D \left<U_1   1_3\right>+ b_B  b_D \left<U_1^2\right>+ b_B  c_D \left<U_1   U_2\right>+ b_B  d_D \left<U_1   r\right>+ b_B  e_D \left<U_1   s\right>\\
&\hphantom{=}+ c_B  a_D \left<U_2   1_3\right>+c_B  b_D \left<U_2   U_1\right>+ c_B  c_D \left<U_2^2\right>+ c_B  d_D \left<U_2   r\right>+ c_B  e_D \left<U_2   s\right>\\
&\hphantom{=}+ d_B  a_D \left<r   1_3\right>+ d_B  b_D \left<r   U_1\right>+ d_B  c_D \left<r   U_2\right>+ d_B  d_D \left<r^2\right>+ d_B  e_D \left<r   s\right>\\
&\hphantom{=}+ e_B  a_D \left<s   1_3\right>+ e_B  b_D \left<s   U_1\right>+ e_B  c_D \left<s   U_2\right>+ e_B  d_D \left<sr\right>+ e_B  e_D \left<s^2\right>.
\end{align*}
The brackets for the pairs in the right-hand side can be evaluated by applying the  following multiplication table. 
\begin{table}[H]
\centering
\begin{tabular}{|c|c|c|c|c|c|}
\hline 
$.$ & $1_3$ & $U_{1}$ & $U_{2}$ & $r$ & $s$\tabularnewline
\hline 
\hline 
$1_3$ & $1_3$ & $U_{1}$ & $U_{2}$ & $r$ & $s$\tabularnewline
\hline 
$U_{1}$ & $U_{1}$ & $\bigcirc\sqcup U_{1}$ & $s$ & $U_{1}$ & $\bigcirc\sqcup s$\tabularnewline
\hline 
$U_{2}$ & $U_{2}$ & $r$ & $\bigcirc\sqcup U_{2}$ & $\bigcirc \sqcup r$ & $U_{2}$\tabularnewline
\hline 
$r$ & $r$ & $\bigcirc\sqcup r$ & $U_{2}$ & $r$ & $\bigcirc\sqcup U_{2}$\tabularnewline
\hline 
$s$ & $s$ & $U_{1}$ & $\bigcirc \sqcup s$ & $\bigcirc \sqcup U_{1}$ & $s$\tabularnewline
\hline 
\end{tabular}
\caption{Multiplication of elements in $ \mathcal{D}_3 $.}
\label{tab:product}
\end{table}
The proof is then completed by factoring with respect to the resulting brackets, eventually simplified according to $ (\mathbf{K2}) $.
\end{proof}

\begin{notation}
Let $ B_n:=BB\cdots B $ denote the $ 3 $-tangle obtained by multiplying the $ 3 $-tangle $ B $  $ n$ times, with $ B_0:=1_3 $. For convenience, we shall identify the bracket formal expression in \eqref{eq:bracket} by the  $ 5 $-tuple $ [a,b,c,d,e]^T$. Similarly, assume that $\left<B_n\right>$ is identified by $ [a_n,b_n,c_n,d_n,e_n]^T $.
\end{notation}

\begin{lemma} The bracket $ 5 $-tuple for $ B_n $ is given by
\begin{equation}\label{eq:matrixdefinition}
\begin{bmatrix}
a_n\\b_n\\c_n\\d_n\\e_n
\end{bmatrix}=\begin{bmatrix}a & 0 & 0 & 0 & 0\\
b & a+b x+d & 0 & 0 & d x+b\\
c & 0 & a+c x+e& c+e x & 0\\
d & 0 & d x+b & a+b x+d & 0\\
e & c+e x & 0 & 0 & a+c x+e\end{bmatrix}^n\begin{bmatrix}
1\\0\\0\\0\\0
\end{bmatrix}.
\end{equation}
\end{lemma}

\begin{proof}\label{lem:lem1} We write $ B_{n+1}=BB_n $, then from \Lem{lem:product} we have
\begin{equation}\label{eq:matrix}
\begin{bmatrix}
a_{n+1}\\b_{n+1}\\c_{n+1}\\d_{n+1}\\e_{n+1}
\end{bmatrix}=\begin{bmatrix}a & 0 & 0 & 0 & 0\\
b & a+b x+d & 0 & 0 & d x+b\\
c & 0 & a+c x+e& c+e x & 0\\
d & 0 & d x+b & a+b x+d & 0\\
e & c+e x & 0 & 0 & a+c x+e\end{bmatrix}\begin{bmatrix}
a_n\\b_n\\c_n\\d_n\\e_n
\end{bmatrix}.
\end{equation}
We conclude by unfolding the recurrence and taking into consideration  the initial condition $ [a_0,b_0,c_0,d_0,e_0]^T=[1,0,0,0,0]^T$.
\end{proof}

We let $ M_B $ denote the  $ 5\times5 $ matrix in \eqref{eq:matrixdefinition}, and we will later refer to it as the \textit{states matrix} for the $ 3 $-tangle $ B$. Using the standard method for computing \eqref{eq:matrixdefinition} we obtain the characteristic polynomial for $ M_B $
\[
\chi\left(M_B,\lambda\right)=-(\lambda-a)\left(\lambda-\dfrac{1}{2}\left(p-q\right)\right)^2\left(\lambda-\dfrac{1}{2}\left(p+q\right)\right)^2,
\] 
then
\begin{align}
a_n &= a^n,\label{eq:an}\\
b_n&=\dfrac{-1}{2q\left(x^2-1\right)}\left(2 a^n q x+\left(\dfrac{p-q}{2}\right)^n \left((b-c) x^2+(-d-e-q) x-2 b\right)\right.\nonumber\\
&\hphantom{=+\dfrac{-1}{2q(x^2-1)}}\left.+\left(\dfrac{p+q}{2}\right)^n \left((-b+c) x^2+(d+e-q) x+2 b\right)  \right),\\
c_n&=\dfrac{-1}{2q\left(x^2-1\right)}\left(2 a^n q x+\left(\dfrac{p+q}{2}\right)^n \left((b-c) x^2+(d+e-q) x+2 c\right)\right.\nonumber\\
&\hphantom{=+\dfrac{-1}{2q(x^2-1)}}\left.+\left(\dfrac{p-q}{2}\right)^n \left((-b+c) x^2+(-d-e-q) x-2 c\right)\right),\\
d_n&=\dfrac{1}{2q\left(x^2-1\right)}\left(2 a^n q+\left(\dfrac{p-q}{2}\right)^n \left(-2 d x^2+(-b-c) x+d-e-q\right)\right.\nonumber\\
&\hphantom{=+\dfrac{-1}{2q\left(x^2-1\right)} }\left.+\left(\dfrac{p+q}{2}\right)^n \left(2 d x^2+\left(b+c\right) x-d+e-q\right)\right),\\
e_n&=\dfrac{1}{2q\left(x^2-1\right)}\left(2 a^n q+\left(\dfrac{p-q}{2}\right)^n \left(-2 e x^2+(-b-c) x-d+e-q\right)\right.\nonumber\\
&\hphantom{=+\dfrac{1}{2q\left(x^2-1\right)}}\left.+\left(\dfrac{p+q}{2}\right)^n \left(2 e x^2 +(b+c) x +d-e-q\right)\right),\label{eq:en}
\end{align}
where
\begin{align}
p&:=(b+c)x+2 a+d+e\label{eq:p},\\
q&:=\sqrt{\left(b^2 -2 b c +c^2 +4 d e\right)x^2 + \left(2 b d +2 c d +2 b e +2 c e\right) x +  4 b c+d^2-2 d e+e^2}.\label{eq:q}
\end{align}
 
Now let $ \overline{B_n} $ denote the tangle closure of $ B_n $. In order to evaluate $ \left<\overline{B_n}\right> $ from  formula \eqref{eq:matrixdefinition} we need to apply the
closure to the elements of $ \mathcal{D}_3 $.

\begin{lemma}
The expression of the bracket polynomial for the closure $ \overline{B_n} $ is given by 
\begin{equation}\label{eq:closure}
\left<\overline{B_n}\right>= x^3a_n+x^2\left(b_n+c_n\right)+x\left(d_n+e_n\right).
\end{equation}
\end{lemma}

The splitting at each crossing do not conflict with the closing process, hence the only point remaining concerns the evaluation of the brackets to the closure of the elements of $ \mathcal{D}_3 $, namely
\[\begin{array}{ccccccccc}
\protect\includegraphics[width=0.1\linewidth,valign=c]{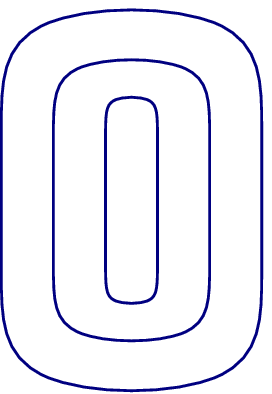}&\qquad&
\protect\includegraphics[width=0.1\linewidth,valign=c]{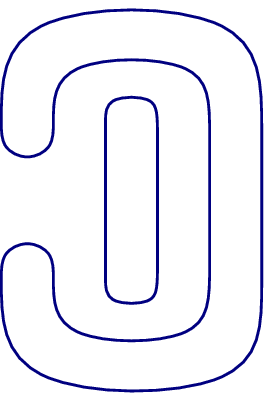}&\qquad&
\protect\includegraphics[width=0.1\linewidth,valign=c]{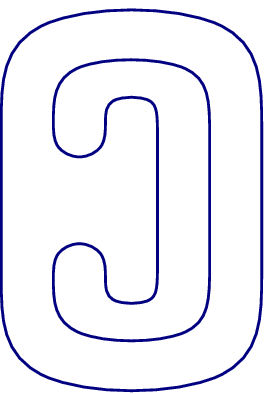}&\qquad&
\protect\includegraphics[width=0.1\linewidth,valign=c]{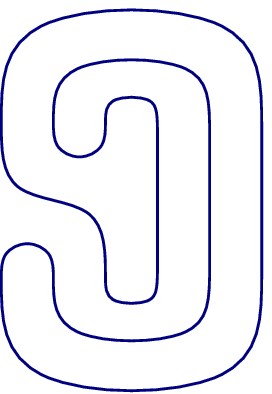}&\qquad&
\protect\includegraphics[width=0.1\linewidth,valign=c]{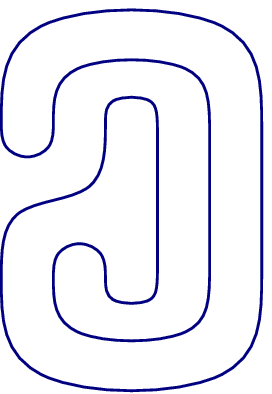} \\
\left<\overline{ 1_3}\right>=x^3  &\qquad&  \left<\overline{ U_1}\right>=x^2  &\qquad&  \left<\overline{U_2} \right>=x^2	&\qquad&  \left<\overline{r}\right>=x &\qquad& \left< \overline{s} \right>=x
\end{array}.\]

Next, combining  \eqref{eq:matrixdefinition}, \eqref{eq:an}--\eqref{eq:en} and \eqref{eq:closure}, we obtain a better expression of the bracket:

\begin{lemma}
The bracket polynomial for  the  knot $ \overline{B_n} $ is given by 
\begin{equation}\label{eq:bracketpolynomial}
\left<\overline{B_n}\right>=xa^n\left(x^2-2\right)+x\left(\left(\dfrac{p-q}{2}\right)^n+\left(\dfrac{p+q}{2}\right)^n\right),
\end{equation}
where $ p $ and $ q $ are expressions defined in \eqref{eq:p} and \eqref{eq:q}.
\end{lemma}
Finally, we let $ \overline{B}(x;y):=\sum_{n\geq0}^{}\left<\overline{B_n}\right>y^n $ denote the generating function of  $ \big(\left<\overline{B_n}\right>\big)_n $. By \eqref{eq:bracketpolynomial} we deduce
\begin{align*}
\overline{B}(x;y)&=\resizebox{.888\linewidth}{!}{$\frac{\left( (b+c)x+2 a+d+e\right)  y-2}{\left( \left( d e-b c\right)    {{x}^{2}}+\left( -a c-a b\right)  x+\left( -d-a\right)  e-a d+b c-{{a}^{2}}\right)    {{y}^{2}}+\left( (b+c)x+2 a+d+e\right)  y-1}$}\\
&\hphantom{=}+\frac{x \left( {{x}^{2}}-2\right) }{1-a y}.
\end{align*}

\section{Application}\label{sec:application}
Throughout this section, let us refer to the $ 3 $-tangles in \eqref{eq:braid} as \textit{generators}. Recall that in the expression $ \left<\overline{B_n}\right> = \sum_{k>0}^{}s_B(n,k)x^k $  we have $ b_{n,k}=\#\{S\mid \textit{$ S $ is a state of $ B_n $  and  $ |S|=k $}\} $, with $ B\in\{T,C,E\} $. For each flat sinnet Turk's head below, we will give the corresponding distribution  $ \left(s_B{(n,k)}\right)_{n,k}$ for small values of $ n $ and $ k $.

\begin{enumerate}
\item \textbf{Three-lead Turk's head}. Let $ \sum_{k\geq 0}^{}s_T(n,k)x^k:=\left<\overline{T_n}\right> $.

\begin{itemize}
\item Bracket for the generator $ T $:

\begin{align*}
\left<\protect\includegraphics[width=0.045\linewidth,valign=c]{genB1}\right>&=\left<\protect\includegraphics[width=0.045\linewidth,valign=c]{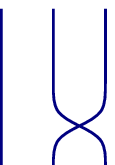}\right>+\left<\protect\includegraphics[width=0.045\linewidth,valign=c]{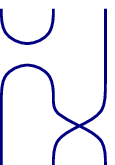}\right>\\
&=\left<\protect\includegraphics[width=0.045\linewidth,valign=c]{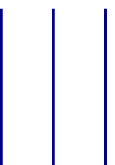}\right>+\left<\protect\includegraphics[width=0.045\linewidth,valign=c]{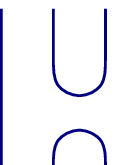}\right>+\left<\protect\includegraphics[width=0.045\linewidth,valign=c]{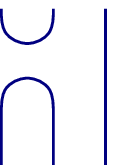}\right>+\left<\protect\includegraphics[width=0.045\linewidth,valign=c]{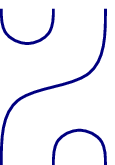}\right>\\
\left<T\right>&=\left<1_3\right>+\left<U_1\right>+\left<U_2\right>+\left<s\right>.
\end{align*}

\item States matrix:
\[M_T=\begin{bmatrix}1 & 0 & 0 & 0 & 0\\
1 & x+1 & 0 & 0 & 1\\
1 & 0 & x+2 & x+1 & 0\\
0 & 0 & 1 & x+1 & 0\\
1 & x+1 & 0 & 0 & x+2\end{bmatrix}.\]

\item Bracket for $ T_n $:
\[\left<\overline{T_n}\right>=x\left(x^2-2 \right)+  x\left(\left(\dfrac{2 x+3-\sqrt{4 x+5}}{2}\right)^n+ \left(\dfrac{2x+3+\sqrt{4 x+5}}{2}\right)^n\right).\]

\item Generating function:
\[
\overline{T}(x;y)=\dfrac{x   \left( \left( -2 x-3\right)  y+2\right) }{\left( {{x}^{2}}+2 x+1\right)    {{y}^{2}}+\left( -2 x-3\right)  y+1}+\dfrac{x   \left( {{x}^{2}}-2\right) }{1-y}.
\]
\item Distribution of $  \left(s_T{(n,k)}\right)_{n,k}$: \cite[\seqnum{A316659}]{Sloane}
\begin{table}[H]
\centering
\resizebox{\linewidth}{!}{%
$\begin{array}{c|rrrrrrrrrrrrrrrrrrrrrrrrr}
n\ \backslash\ k		 &0		 &1		 &2		 &3		 &4		 &5		 &6		 &7		&8	&9	&10&11\\
\midrule
0 & 0 & 0 & 0 & 1\\
1 & 0 & 1 & 2 & 1\\
2 & 0 & 5 & 8 & 3\\
3 & 0 & 16 & 30 & 16 & 2\\
4 & 0 & 45 & 104 & 81 & 24 & 2\\
5 & 0 & 121 & 340 & 356 & 170 & 35 & 2\\
6 & 0 & 320 & 1068 & 1411 & 932 & 315 & 48 & 2\\
7 & 0 & 841 & 3262 & 5209 & 4396 & 2079 & 532 & 63 & 2\\
8 & 0 & 2205 & 9760 & 18281 & 18784 & 11440 & 4144 & 840 & 80 & 2\\
9 & 0 & 5776 & 28746 & 61786 & 74838 & 55809 & 26226 & 7602 & 1260 & 99 & 2\\
10 & 0 & 15125 & 83620 & 202841 & 282980 & 249815 & 144488 & 54690 & 13080 & 1815 & 120 & 2\\
\end{array}$}
\caption{Values of $ s_T(n,k) $ for $ 0\leq n\leq 10 $ and $ 0\leq k\leq 11 $.}
\label{tab:T}
\end{table}
\end{itemize}

\item \textbf{Chain sinnet}. Let $ \sum_{k\geq 0}^{}s_C(n,k)x^k:=\left<\overline{C_n}\right> $.
\begin{itemize}
\item Bracket for the generator $ C $:
\begin{align*}\label{key}
\left<\protect\includegraphics[width=0.045\linewidth,valign=c]{genB2}\right>&=\left<\protect\includegraphics[width=0.045\linewidth,valign=c]{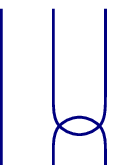}\right>+\left<\protect\includegraphics[width=0.045\linewidth,valign=c]{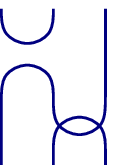}\right>\\
&=\left<\protect\includegraphics[width=0.045\linewidth,valign=c]{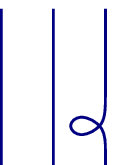}\right>+\left<\protect\includegraphics[width=0.045\linewidth,valign=c]{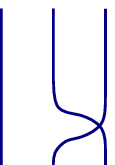}\right>+\left<\protect\includegraphics[width=0.045\linewidth,valign=c]{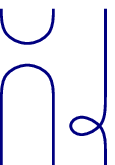}\right>+\left<\protect\includegraphics[width=0.045\linewidth,valign=c]{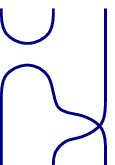}\right>\\
&=\left<\protect\includegraphics[width=0.045\linewidth,valign=c]{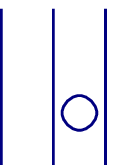}\right>
+\left<\protect\includegraphics[width=0.045\linewidth,valign=c]{T24}\right>
+\left<\protect\includegraphics[width=0.045\linewidth,valign=c]{T24}\right>
+\left<\protect\includegraphics[width=0.045\linewidth,valign=c]{T23}\right>
+\left<\protect\includegraphics[width=0.045\linewidth,valign=c]{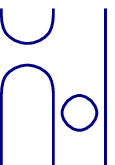}\right>
+\left<\protect\includegraphics[width=0.045\linewidth,valign=c]{T22}\right>
+\left<\protect\includegraphics[width=0.045\linewidth,valign=c]{T22}\right>
+\left<\protect\includegraphics[width=0.045\linewidth,valign=c]{T21}\right>\\
\left<C\right>&=(x+2)\left<1_3\right>+(x+2)\left<U_1\right>+\left<U_2\right>+\left<s\right>.
\end{align*}

\item States matrix:
\[M_C=\begin{bmatrix}x+2 & 0 & 0 & 0 & 0\\
x+2 & {{x}^{2}}+3 x+2 & 0 & 0 & x+2\\
1 & 0 & 2 x+3 & x+1 & 0\\
0 & 0 & x+2 & {{x}^{2}}+3 x+2 & 0\\
1 & x+1 & 0 & 0 & 2 x+3\end{bmatrix}.\]

\item Bracket for $ C_n $:
\begin{align*}
\left<\overline{C_n}\right>&=x\left(x^2-2\right)(x+2)^n+
x\left(\left(\dfrac{x^2+5x+5-\sqrt{{{x}^{4}}+2 {{x}^{3}}+3 {{x}^{2}}+10 x+9}}{2}\right)^n\right.\\ &\hphantom{=}+\left.\left(\dfrac{x^2+5x+5+\sqrt{{{x}^{4}}+2 {{x}^{3}}+3 {{x}^{2}}+10 x+9}}{2}\right)^n\right).
\end{align*}

\item Generating function
\[\overline{C}(x;y)=\frac{x   \left( \left( -{{x}^{2}}-5 x-5\right)  y+2\right) }{\left( 2 {{x}^{3}}+8 {{x}^{2}}+10 x+4\right)    {{y}^{2}}+\left( -{{x}^{2}}-5 x-5\right)  y+1}+\frac{x   \left( {{x}^{2}}-2\right) }{1-\left( x+2\right)  y}.\]

\item Distribution of $  \left(s_C(n,k)\right)_{n,k}$: 
\begin{table}[H]
\centering
\resizebox{\linewidth}{!}{%
$\begin{array}{c|rrrrrrrrrrrrrrrrrrrrrrrrr}
n\ \backslash\ k		 &0		 &1		 &2		 &3		 &4		 &5		 &6		 &7		&8	&9	&10&11&12&13\\
\midrule
0 & 0 & 0 & 0 & 1\\
1 & 0 & 1 & 3 & 3 & 1\\
2 & 0 & 9 & 22 & 21 & 10 & 2\\
3 & 0 & 49 & 141 & 164 & 105 & 42 & 10 & 1\\
4 & 0 & 225 & 796 & 1186 & 1008 & 569 & 232 & 67 & 12 & 1\\
5 &0 & 961 & 4115 & 7677 & 8400 & 6205 & 3393 & 1435 & 461 & 105 & 15 & 1\\
6 & 0 & 3969 & 20106 & 45481 & 61630 & 57078 & 39298 & 21239 & 9198 & 3151 & 822 & 153 & 18 & 1
\end{array}$}
\caption{Values of $ s_C(n,k) $ for $ 0\leq n\leq 6 $ and $ 0\leq k\leq 13 $.}
\label{tab:C}
\end{table}
\end{itemize}

\item \textbf{Figure-eight chain}. Let $ \sum_{k\geq 0}^{}s_E(n,k)x^k:=\left<\overline{E_n}\right> $.
\begin{itemize}
\item Bracket for the generator $ E $:
\begin{align*}\label{key}
\left<\protect\includegraphics[width=0.045\linewidth,valign=c]{genB3}\right>&=\left<\protect\includegraphics[width=0.045\linewidth,valign=c]{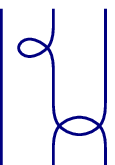}\right>+\left<\protect\includegraphics[width=0.045\linewidth,valign=c]{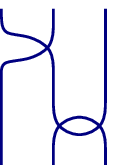}\right>\\
&=\left<\protect\includegraphics[width=0.045\linewidth,valign=c]{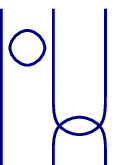}\right>+\left<\protect\includegraphics[width=0.045\linewidth,valign=c]{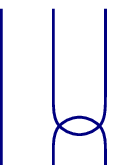}\right>+\left<\protect\includegraphics[width=0.045\linewidth,valign=c]{genB2}\right>=(x+1)\left<\protect\includegraphics[width=0.045\linewidth,valign=c]{E21}\right>+\left<C\right>\\
&=(x+1)\left(\left<\protect\includegraphics[width=0.045\linewidth,valign=c]{C38}\right>
+\left<\protect\includegraphics[width=0.045\linewidth,valign=c]{T24}\right>
+\left<\protect\includegraphics[width=0.045\linewidth,valign=c]{T24}\right>
+\left<\protect\includegraphics[width=0.045\linewidth,valign=c]{T23}\right>\right)+\left<C\right>\\
\left<E\right>&=\left(x^2+4x+4\right)\left<1_3\right>+(x+2)\left<U_1\right>+(x+2)\left<U_2\right>+\left<s\right>.
\end{align*}

\item States matrix:
\[M_E=\begin{bmatrix}{{x}^{2}}+4 x+4 & 0 & 0 & 0 & 0\\
x+2 & 2 {{x}^{2}}+6 x+4 & 0 & 0 & x+2\\
x+2 & 0 & 2 {{x}^{2}}+6 x+5 & 2 x+2 & 0\\
0 & 0 & x+2 & 2 {{x}^{2}}+6 x+4 & 0\\
1 & 2 x+2 & 0 & 0 & 2 {{x}^{2}}+6 x+5\end{bmatrix}.\]

\item Bracket for $ \overline{E_n} $:
\begin{align*}
\left<\overline{E_n}\right>&=x(x^2-2)\left(x^2+4x+4\right)^n+ x\left(\left(\dfrac{4 {{x}^{2}}+12 x+9-\sqrt{8 {{x}^{2}}+24 x+17}}{2}\right)^n\right.\\
&\hphantom{=}+\left. \left(\dfrac{4 {{x}^{2}}+12 x+9+\sqrt{8 {{x}^{2}}+24 x+17}}{2}\right)^n\right).
\end{align*}

\item Generating function
\begin{align*}
\overline{E}(x;y)&=\frac{x   \left( \left( -4 {{x}^{2}}-12 x-9\right)  y+2\right) }{\left( 4 {{x}^{4}}+24 {{x}^{3}}+52 {{x}^{2}}+48 x+16\right)    {{y}^{2}}+\left( -4 {{x}^{2}}-12 x-9\right)  y+1}\\
&\hphantom{=}+\frac{x   \left( {{x}^{2}}-2\right) }{1-{{\left( x^2+4x+4\right) }} y}.
\end{align*}
\item Distribution of $  \left(s_E(n,k)\right)_{n,k}$:
\begin{table}[H]
\centering
\resizebox{\linewidth}{!}{%
$\begin{array}{c|rrrrrrrrrrrrrrrrrrrrrrrrr}
n\ \backslash\ k		 &0		 &1		 &2		 &3		 &4		 &5		 &6		 &7		&8	&9	&10&11&12&13\\
\midrule
0 & 0 & 0 & 0 & 1\\
1 & 0 & 1 & 4 & 6 & 4 & 1\\
2 & 0 & 17 & 56 & 80 & 64 & 30 & 8 & 1\\
3 & 0 & 169 & 660 & 1120 & 1096 & 684 & 280 & 74 & 12 & 1\\
4 & 0 & 1377 & 6640 & 14112 & 17504 & 14128 & 7808 & 3008 & 800 & 142 & 16 & 1\\
5 & 0 & 10201 & 59660 & 156624 & 244280 & 252460 & 182544 & 94960 & 35904 & 9800 & 1880 & 242 & 20 & 1
\end{array}$}
\caption{Values of $ s_E(n,k) $ for $ 0\leq n\leq 5 $ and $ 0\leq k\leq 13 $.}
\label{tab:E}
\end{table}
\end{itemize}
\end{enumerate}

\begin{remark}
Column $ 1 $ in \Tabs{tab:T} is sequence \seqnum{A004146} in the OEIS \cite{Sloane}, the sequence of alternate Lucas numbers minus $ 2 $, which is  the determinant of the Turk's Head Knots $ THK(3,n) $ \cite{KST}. Column $ 2 $ is the $ x $-coefficients of a generalized  Jaco-Lucas polynomials for even indices \cite{Sun} (see column $ 1 $ in triangle \seqnum{A122076})  and is also a subsequence of a Fibonacci-Lucas convolution \seqnum{A099920} for odd indices. Column $ 1 $ in 	\Tabs{tab:C}  is \seqnum{A060867} with a leading $ 0 .$ 

Rows $ 1 $ in \Tabs{tab:T}, \Tabs{tab:C}, \Tabs{tab:E} match the coefficients of the bracket for the $ 2 $-twist loop (see row $ 1 $ in \seqnum{A300184}, \seqnum{A300192} and row $ 0 $ in \seqnum{A300454}), the $ 3 $-twist loop and the $4 $-twist loop modulo planar isotopy and move on the $ 2 $-sphere \cite{Ramaharo}, respectively (see \Figs{Fig:Equivalent} \subref{subfig:T1}, \subref{subfig:C1} and \subref{subfig:E1}). Row $ 2 $ in \Tabs{tab:T} gives those of the figure-eight knot (see \Figs{Fig:Equivalent} \subref{subfig:C1} and row $ 1 $ in \seqnum{A300454}).

\begin{figure}[H]
\centering
{ %
\subfigure[$T_1$]{\includegraphics[width=0.22\linewidth]{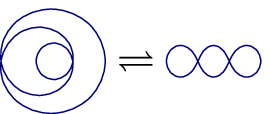}\label{subfig:T1}}\hfill%
\subfigure[$T_2$]{\includegraphics[width=0.22\linewidth]{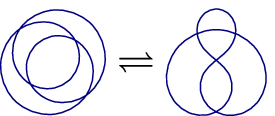}\label{subfig:C1}}\hfill%
\subfigure[$C_1$]{\includegraphics[width=0.23\linewidth]{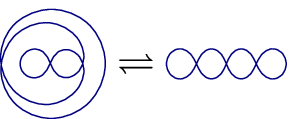}\label{subfig:F1}}\hfill%
\subfigure[$E_1$]{\includegraphics[width=0.25\linewidth]{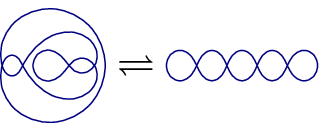}\label{subfig:E1}} %
}
\caption{Equivalent knot shadow diagrams.}
\label{Fig:Equivalent}
\end{figure}
\end{remark}

\bigskip
\small \textbf{2010 Mathematics Subject Classifications}:  05A19;  57M25.
\end{document}